\documentclass[11pt]{amsart}

\usepackage{amsthm,amsmath,amssymb,amscd,amsfonts}
\usepackage{mathrsfs}
\theoremstyle{plain}
\newtheorem{theorem}{Theorem}[section]

\newtheorem{proposition}[theorem]{Proposition}

\newtheorem{corollary}[theorem]{Corollary}
\newtheorem{def-thm}[theorem]{Definition-Theorem}
\newtheorem{lemma}[theorem]{Lemma}

\newtheorem{definition}[theorem]{Definition}

\newtheorem*{tha}{Theorem A}
\newtheorem*{thb}{Theorem B}

\newtheorem*{thm}{Main Theorem}

\theoremstyle{definition}
\newtheorem{remark}[theorem]{Remark}

\newcommand{\PP}{\mathbb{P}}

\newcommand{\QQ}{\mathbb{Q}}

\newcommand{\exc}{{\operatorname{exc}}}

\DeclareMathOperator{\ord}{ord}

\DeclareMathOperator{\Div}{Div}
\DeclareMathOperator{\CDiv}{CDiv}

\DeclareMathOperator{\lcm}{lcm}
\DeclareMathOperator{\Supp}{Supp}

\DeclareMathOperator{\Nev}{Nev}

\DeclareMathOperator{\Nevbir}{Nev_{\operatorname{bir}}}
\DeclareMathOperator{\NevEF}{Nev_{\operatorname{EF}}}

\DeclareMathSymbol{\idot}{\mathbin}{operators}{`\.}
\allowdisplaybreaks
\begin{document}
\title{An Evertse--Ferretti Nevanlinna constant and its consequences}

\author{Min Ru}

\address{Department of Mathematics\\University of Houston\\4800 Calhoun Road\\
  Houston, TX 77204\\USA}

\email{minru@math.uh.edu}

\thanks{The first author is supported in part by the Simons Foundation
  grant award \#531604.}

\author{Paul Vojta}

\address{Department of Mathematics\\University of California\\
  970 Evans Hall \#3840\\Berkeley, CA  94720-3840\\USA}

\email{vojta@math.berkeley.edu}

\thanks{Parts of this work were performed while the second author enjoyed
  the kind hospitality of the Fields Institute, partially supported by
  NSF grant DMS-0753152.}


\begin{abstract}
In this paper, we introduce the notion of an Evertse--Ferretti
Nevanlinna constant and compare it with the birational Nevanlinna constant
introduced by the authors in a recent joint paper.
We then use it to recover several previously known results.
This includes a 1999 example of Faltings from his \emph{Baker's Garden}
article.
\end{abstract}

\maketitle

\section{Introduction}\label{intro}

Let $X$ be a projective variety and let $D$ be an effective Cartier divisor
on $X$.  Let $\mathscr L$ be a line sheaf on $X$
with $\dim H^0(X, \mathscr L^N) > 1$ for some $N>0$.
Other notations are as described in Section \ref{nev_ef} and in \cite{RV}.

In \cite{RV}, the authors introduced the birational Nevanlinna constant
$\Nevbir(\mathscr L, D)$, which extends the notion of the
Nevanlinna constant $\Nev(\mathscr L, D)$ introduced by
the first-named author \cite{ru}, and proved the following results.

\begin{tha}\label{b_thmc}
Let $k$ be a number field, and let $S$ be a finite set of places of $k$
containing all archimedean places.
Let $X$ be a projective variety over $k$, let $D$ be an effective Cartier
divisor on $X$, and let $\mathscr L$ be a line sheaf on $X$
with $\dim H^0(X,\mathscr L^N) > 1$ for some $N>0$.
Then, for every $\epsilon>0$,
there is a proper Zariski-closed subset $Z$ of $X$ such that the inequality
\begin{equation}
  m_S(x, D) \le \left(\Nevbir(\mathscr L, D)+\epsilon\right) h_{\mathscr L}(x)
\end{equation}
holds for all $x\in X(k)$ outside of $Z$.
\end{tha}

\begin{thb}\label{b_thmd}
Let $X$ be a complex projective variety, let $D$ be an effective Cartier
divisor, and let $\mathscr L$ be a line sheaf on $X$
with $\dim H^0(X, \mathscr L^N) > 1$ for some $N>0$.
Let $f\colon\mathbb C\to X$ be a holomorphic mapping with Zariski-dense image.
Then, for every $\epsilon>0$,
\begin{equation}\label{b_thmc_eq1}
  m_f(r, D) \le_\exc
    \left(\Nevbir(\mathscr L, D)+\epsilon\right) T_{f, \mathscr L}(r)\;.
\end{equation}
\end{thb}

Recall that $\Nevbir(\mathscr L, D)$ and $\Nev(\mathscr L, D)$
are defined as follows.  (In \cite{RV} we defined them only for
$\mathbb Q$-Cartier divisors and required $\mu$ to be rational, but
we are working over $\mathbb R$ in this paper.
See Section \ref{r_divisors} for more details on $\mathbb R$-Cartier divisors.)

\begin{definition}\label{def_nevbir}
Let $X$ be a complete variety, let $D$ be an effective $\mathbb R$-Cartier
divisor on $X$, and let $\mathscr L$ be a line sheaf on $X$.  Then we define
$$\Nevbir(\mathscr L,D) = \inf_{N,V,\mu} \frac{\dim V}{\mu}\;,$$
where the infimum passes over all triples $(N,V,\mu)$ such that
$N\in\mathbb Z_{>0}$, $V$ is a linear subspace of $H^0(X,\mathscr L^N)$
with $\dim V>1$, and $\mu\in\mathbb R_{>0}$, with the following property.
There exist a variety $Y$ and a proper birational morphism $\phi\colon Y\to X$
such that the following condition holds.
For all $Q\in Y$ there is a basis $\mathcal B$ of $V$ such that
\begin{equation}\label{eq3}
  \phi^{*}(\mathcal B) \ge \mu N\phi^{*}D
\end{equation}
in a Zariski-open neighborhood $U$ of $Q$,
relative to the cone of effective $\mathbb R$-Cartier divisors on $U$.
If there are no such triples $(N,V,\mu)$, then $\Nevbir(\mathscr L,D)$
is defined to be $+\infty$.

If $L$ is a Cartier divisor or Cartier divisor class on $X$, then
we define $\Nevbir(L,D)=\Nevbir(\mathscr O(L),D)$.
We also define $\Nevbir(D)=\Nevbir(D,D)$.
\end{definition}

The Nevanlinna constant $\Nev(\mathscr L, D)$ is defined similarly
but without taking the birational model (i.e., it requires $Y$ to be the
normalization of $X$).
Obviously, $\Nevbir(\mathscr L,D) \leq \Nev(\mathscr L,D)$.





The purpose of this note is to introduce the following variant of the
Nevanlinna constant, based on a theorem of Evertse and Ferretti
\cite{ef_festschrift}.

\begin{definition}\label{def_nevef}
Let $X$ be a complete variety, let $D$ be an effective $\mathbb R$-Cartier
divisor on $X$, and let $\mathscr L$ be a line sheaf on $X$.  Then we define
$$\NevEF(\mathscr L,D) = \inf_{N,\mu} \frac{\dim X + 1}{\mu}\;,$$
where the infimum passes over all pairs $(N,\mu)$ such that
$N\in\mathbb Z_{>0}$ and $\mu\in\mathbb R_{>0}$, with the following property.
There exist a variety $Y$ and a proper birational morphism $\phi\colon Y\to X$
such that the following condition holds.
For all $Q\in Y$ there exist a base-point-free linear subspace
$V\subseteq H^0(X,\mathscr L^N)$ with $\dim V=\dim X+1$ and
a basis $\mathcal B$ of $V$ such that
\begin{equation}
  \phi^{*}(\mathcal B) \ge \mu N\phi^{*}D
\end{equation}
in a Zariski-open neighborhood $U$ of $Q$,
relative to the cone of effective $\mathbb R$-divisors on $U$.
If there are no such pairs $(N,\mu)$, then $\NevEF(\mathscr L,D)$
is defined to be $+\infty$.

If $L$ is a Cartier divisor or Cartier divisor class on $X$, then
we define $\NevEF(L,D)=\NevEF(\mathscr O(L),D)$.
We also define $\NevEF(D)=\NevEF(D,D)$.
\end{definition}

Our main result is as follows.

\begin{thm}
Let $X$ be a geometrically integral variety over a number field $k$,
let $D$ be an effective
$\mathbb R$-Cartier divisor on $X$, and let $\mathscr L$ be a line sheaf on $X$.
Then
\begin{equation}\label{ineq_b_nev_ef_nev}
  \Nevbir(\mathscr L,D) \le \NevEF(\mathscr L,D)\;.
\end{equation}
\end{thm}

Combining this with Theorem A and Lemma \ref{lemma_non_geom_irred}, we then have

\begin{theorem}\label{ef_thmc}
Let $k$ be a number field, and let $S$ be a finite set of places of $k$
containing all archimedean places.
Let $X$ be a projective variety over $k$, and let $D$ be an effective
$\mathbb R$-Cartier divisor on $X$.  Then, for every $\epsilon>0$,
there is a proper Zariski-closed subset $Z$ of $X$ such that the inequality
\begin{equation}
  m_S(x, D) \le \left(\NevEF(D)+\epsilon\right) h_D(x)
\end{equation}
holds for all $x\in X(k)$ outside of $Z$.
\end{theorem}

\begin{corollary}\label{cor_nev_ef_int_pts}
Let $X$ be a projective variety, and let $D$ be an ample Cartier divisor
on $X$.  If $\NevEF(D)<1$ then there is a proper Zariski-closed subset $Z$
of $X$ such that any set of $D$-integral points on $X$ has only finitely many
points outside of $Z$.
\end{corollary}

A similar theorem and corollary also hold in the case of holomorphic curves.

Theorem \ref{ef_thmc} generalizes the original theorem of Evertse and Ferretti,
and we show how to recover the latter theorem from Theorem \ref{ef_thmc}
(see Theorem \ref{thm_ef}).

Some other applications are also given.  This includes discussion of
a 1999 class of examples obtained by Faltings \cite{faltings}.

These examples consist of irreducible divisors $D$ on $\mathbb P^2$
for which $\mathbb P^2\setminus D$ has only finitely many integral points
over any number ring, and over any localization of such a ring
away from finitely many places.

Faltings' paper is notable for two reasons.
First, the divisor $D$ is irreducible.
Prior to the paper, the only divisors $D$ on $\mathbb P^2$ for which such
statements were known were divisors with at least four irreducible components.
The second reason is that (as noted at the very end of \cite{faltings})
the paper gives examples of varieties for which
finiteness of integral points is known, yet which cannot be embedded into
semiabelian varieties.  Prior to the paper, the only varieties for which such
finiteness statements were known, and which could not be embedded into
semiabelian varieties, were moduli spaces of abelian varieties.

Faltings' construction was further explored
by Zannier \cite{zannier} using methods of Zannier and Corvaja.
Zannier considered a different family of examples,
which has substantial overlap with the examples of Faltings but does not
contain all of his examples.  After that, Levin \cite[\S\,13]{levin_annals}
derived a generalization, using his method of \emph{large divisors,}
that encompasses the examples of both Faltings and Zannier.

The method of Faltings' paper is to apply the original ``filtration method''
of Faltings and W\"ustholz \cite{fw}, together with their
probabilistic version of Schmidt's Subspace Theorem, to certain
coherent ideal sheaves on a certain \'etale cover of $\mathbb P^2\setminus D$.

This use of ideal sheaves naturally suggested the use of
``generalized Weil functions'' of the second-named author
\cite[\S\,7]{vojta_ssav1},
which, as noted in \cite{RV}, are now more aptly called b-Weil functions.
We first rephrased Faltings' proof in terms of b-Weil functions and replaced
the work of Faltings--W\"ustholz with the method of
Evertse and Ferretti \cite{ef_festschrift}.

In an effort to express this modified proof using the Nevanlinna constant
of the first-named author, we formulated the ``birational Nevanlinna constant''
$\Nevbir(\mathscr L,D)$ (Definition \ref{def_nevbir}).
We found in \cite{RV} that this definition worked well for adapting
Autissier's proof, but that it was not possible to apply it to
Faltings' examples without involving the Evertse--Ferretti method.
This was what led to the formulation of the ``Evertse--Ferretti Nevanlinna
constant'' $\NevEF(\mathscr L,D)$ (Definition \ref{def_nevef}).

The outline of this paper is as follows.  Section \ref{r_divisors} introduces
$\mathbb R$-Cartier b-divisors on a complete variety $X$ and develops some
of their basic properties.
Section \ref{nev_ef} gives more information on the Evertse--Ferretti
Nevanlinna constant, including some equivalent formulations of it,
and proves the Main Theorem.
Section \ref{apps} contains some brief applications of Theorem \ref{ef_thmc},
including the earlier result of Evertse and Ferretti.
Section \ref{falt_sect} concludes this note by discussing a 1999 class
of examples obtained by Faltings \cite{faltings}.

\section{$\mathbb R$-Cartier b-divisors and their b-Weil functions}
\label{r_divisors}

This section introduces some notation for divisors and b-divisors
with coefficients in $\mathbb R$.  (Divisors with real coefficients
were used in \cite{RV}, but were not fully described there
because they were added as a late change to the paper.)

Let $k$ be a field of characteristic zero.  As in \cite{RV}, a \textbf{variety}
over $k$ is an integral scheme, separated and of finite type over $k$.
It is said to be \textbf{complete} if it is proper over $k$.

\subsection{$\mathbb Q$-Cartier and $\mathbb R$-Cartier divisors}

If $X$ is a variety over $k$, then we let $\CDiv(X)$ denote the group of
Cartier divisors on $X$.  As in \cite[1.3.B]{PAGI}, an
\textbf{$\mathbb R$-Cartier divisor} on $X$ is an element of the group
\begin{equation}\label{cdiv_R}
  \CDiv_{\mathbb R}(X) := \CDiv(X)\otimes_{\mathbb Z}\mathbb R\;.
\end{equation}
Thus, every $\mathbb R$-Cartier divisor $D$ on $X$ occurs as a finite sum
$D=\sum c_iD_i$, with $c_i\in\mathbb R$ and $D_i\in\CDiv(X)$ for all $i$.
An $\mathbb R$-Cartier divisor $D$ is \textbf{effective} if it can be
written in the above form with $c_i\ge0$ and $D_i$ effective for all $i$.

If $X$ is normal, then $\CDiv(X)$ is canonically identified with a subgroup
of the group $\Div(X)$ of Weil divisors on $X$, and by
\cite[II~6.3A]{Hartshorne} a Cartier divisor on $X$ is effective if and only if
it is effective as a Weil divisor.  Defining $\CDiv_{\mathbb Q}(X)$,
$\Div_{\mathbb R}(X)$, and $\Div_{\mathbb Q}(X)$ similarly to (\ref{cdiv_R}),
we have that the map $\Div(X)\to\CDiv(X)$ induces natural maps
$\Div_{\mathbb Q}(X)\to\CDiv_{\mathbb Q}(X)$ and
$\Div_{\mathbb R}(X)\to\CDiv_{\mathbb R}(X)$.  These too are injective
($\mathbb Q$ is a localization of $\mathbb Z$, hence is flat over $\mathbb Z$,
and $\mathbb R$ is flat over the field $\mathbb Q$, respectively).
Moreover these maps again preserve the respective cones of effective divisors.

Since the maps $\Div(X)\to\Div_{\mathbb Q}(X)\to\Div_{\mathbb R}(X)$
are injective and respect the notion of effectivity, the maps
$\CDiv(X)\to\CDiv_{\mathbb Q}(X)\to\CDiv_{\mathbb R}(X)$ have these same
properties when $X$ is normal.

If $X$ is not normal, then an effective (integral) Cartier divisor on $X$
remains effective in $\CDiv_{\mathbb Q}(X)$ and $\CDiv_{\mathbb R}(X)$,
but the converse does not hold.  For example, the Cartier divisor $(y/x)$
on the cuspidal cubic curve $y^2=x^3$ is not effective, but it is effective
as an $\mathbb Q$-Cartier or $\mathbb R$-Cartier divisor.

\subsection{B-divisors with coefficients in $\mathbb Z$, $\mathbb Q$,
and $\mathbb R$}

B-divisors consider divisors not only on a given variety, but also on
blowings-up of the variety.  The prefix `b' means \emph{birational.}

Let $X$ be a variety.
Following \cite{bdff}, a \textbf{model} over $X$ is a proper birational
morphism $\pi\colon X_\pi\to X$, modulo isomorphism over $X$.  This will be
denoted $X_\pi$.  The \textbf{Zariski--Riemann} space of $X$ is defined as
$$\mathfrak X = \varprojlim_\pi X_\pi\;,$$
and a \textbf{Cartier b-divisor} on a variety $X$ is an element of the group
$$\CDiv(\mathfrak X) := \varinjlim_\pi\CDiv(X_\pi)\;.$$
Thus, each Cartier b-divisor on $X$ comes from a Cartier divisor $D$
on some model $X_\pi$.  A Cartier b-divisor is said to be \textbf{effective}
if $X_\pi$ and $D$ can be chosen such that $D$ is effective.

An \textbf{$\mathbb R$-Cartier b-divisor} on $X$ is an element of the group
\begin{equation}
  \begin{split}
    \CDiv_{\mathbb R}(\mathfrak X) &:= \CDiv(\mathfrak X)\otimes\mathbb R \\
    &\cong \varinjlim_\pi\CDiv_{\mathbb R}(X_\pi)
  \end{split}
\end{equation}
An $\mathbb R$-Cartier b-divisor on $X$ is said to be \textbf{effective}
if it comes from an effective $\mathbb R$-Cartier divisor
on some model $X_\pi$.  The group $\CDiv_{\mathbb Q}(\mathfrak X)$ of
$\mathbb Q$-Cartier b-divisors on $X$ is defined similarly.

Assume briefly that the variety $X$ is normal.
By \cite[Remark~2.3]{RV}, a Cartier divisor on $X$ is effective
if and only if the corresponding Cartier b-divisor on $X$ is effective.
Therefore the map $\CDiv(X)\to\CDiv(\mathfrak X)$ (from Cartier divisors on $X$
to Cartier b-divisors on $X$) is injective and preserves effectivity.
Likewise, the maps
$\CDiv(\mathfrak X)\to\CDiv_{\mathbb Q}(\mathfrak X)
 \to\CDiv_{\mathbb R}(\mathfrak X)$ are injective and preserve effectivity.

If $X$ is not normal, then an effective Cartier divisor on $X$ still
remains effective as a Cartier b-divisor (or as a $\mathbb Q$-Cartier
or $\mathbb R$-Cartier b-divisor), but not conversely.
The latter is illustrated by the same example as before.

In contrast to (non-birational) Cartier divisors, though,
a Cartier b-divisor on an arbitrary variety $X$ is effective if and only if
it is effective as a $\mathbb Q$-Cartier or $\mathbb R$-Cartier b-divisor.
This is because we can replace $X$ with its normalization without affecting
any of these groups of b-divisors.

In \cite[Prop.~4.10a]{RV}, we showed that the group $\CDiv(\mathfrak X)$
is a \emph{lattice-ordered group}; i.e., an ordered group in which
every two elements have a least upper bound $\mathbf D_1\vee\mathbf D_2$.
This, in turn, relied on \cite[Lemma~4.8]{RV}, which stated that if $D$ is
a Cartier divisor on a variety $Y$, then there is a proper model
$\pi\colon Y'\to Y$ such that $\pi^{*}D$ is a difference of effective
Cartier divisors with disjoint supports.  This lemma remains true for
$\mathbb Q$-Cartier divisors (by multiplying the divisor in question by
a positive integer so as to clear denominators), but is no longer true
for $\mathbb R$-Cartier divisors.
An example of the latter is the $\mathbb R$-divisor $\alpha L_1-\beta L_2$,
where $L_1$ and $L_2$
are distinct lines on $\mathbb P^2$ and $\alpha,\beta>0$ are real numbers
such that $\alpha/\beta$ is irrational.

Therefore $\CDiv(\mathfrak X)_{\mathbb Q}$ is a lattice-ordered group,
but $\CDiv(\mathfrak X)_{\mathbb R}$ is not.  As it turns out, though,
this is not an issue here.  In this paper, the only instances of taking
the least upper bound or greatest lower bound of two b-divisors involve
only integral Cartier b-divisors.

\subsection{B-Weil functions for $\mathbb R$-Cartier b-divisors}

Now assume that $k$ is either a number field or the field $\mathbb C$,
and that $X$ is a complete variety over $k$.
Then a \textbf{b-Weil function} for a $\mathbb R$-Cartier b-divisor on $X$
can be defined by $\mathbb R$-linearity, using the definition of
b-Weil function for (integral) Cartier b-divisors, and these b-Weil functions
can be used as usual to define proximity and counting functions.

Moreover, \cite[Prop.~2.4 and Prop.~4.6]{RV} still hold in this context
(i.e., b-Weil functions for $\mathbb R$-Cartier b-divisors have
the same linearity, existence, and uniqueness properties;
an $\mathbb R$-Cartier b-divisor is effective if and only if
some associated b-Weil function is bounded from below by an $M_k$-constant;
and if $U$ is a nonempty Zariski-open subset of a complete variety
over a number field or over $\mathbb C$, then
any function $\lambda\colon U(M)\to\mathbb R$ extends to a b-Weil function
for at most one $\mathbb R$-Cartier b-divisor on $X$).

In addition, \cite[Lemma~2.5]{RV} extends to $\mathbb R$-Cartier divisors:

\begin{lemma}\label{rv_weil_max_fancy}
Let $X$ be a complete variety, and let $U_1,\dots,U_n$ be Zariski-open
subsets of $X$ that cover $X$.  Let $D_1,\dots,D_n$ be $\mathbb R$-Cartier
divisors on $X$ such that $D_i\big|_{U_i}$ is effective for all $i$, and
let $\lambda_{D_i}$ be Weil functions for $D_i$ for all $i$.
Then there is an $M_k$-constant $\gamma=(\gamma_v)$ such that, for all $v$
and all $x\in X(\overline k_v)$ there is an $i$ such that $x\in U_i$
and $\lambda_{D_i,v}(x)\ge \gamma_v$.
\end{lemma}

\begin{proof}
By taking a finite refinement of $\{U_i\}$, we may assume that each $U_i$ is
affine, and that $D_i\big|_{U_i}=\sum_{j=1}^{n_i} c_{ij}(f_{ij})$ for all $i$,
where $f_{ij}\in\mathscr O(U_i)\setminus\{0\}$ and $c_{ij}\in\mathbb R_{>0}$
for all $i$ and $j$.

Then $\lambda_{D_i}$ is locally $M_k$-bounded from below on $U_i(M_k)$
for all $i$.  Indeed, by \cite[Ch.~10, Prop.~1.3]{lang},
the function $-\sum_j c_j\log\|f_{ij}\|$ is locally $M_k$-bounded from below
on $U_i(M_k)$ for all $i$;
therefore so is $\lambda_{D_i}$ since by definition of Weil function
and N\'eron function the difference between the two functions is
locally $M_k$-bounded.

The proof then concludes as in \cite{RV} (using standard properties
of $M_k$-bounded sets).
\end{proof}

Then \cite[Prop.~3.3]{RV} extends to $\mathbb R$-Cartier b-divisors as follows.

\begin{proposition}\label{rv_prop_nevprime_weil}
Let $X$ be a complete variety over a number field, let $D$ be
an effective $\mathbb R$-Cartier divisor on $X$, let $\mathscr L$ be
a line sheaf on $X$, let $V$ be a linear subspace of $H^0(X,\mathscr L)$
with $\dim V>1$, and let $\mu\in\mathbb R_{>0}$.

Consider the following conditions.
\begin{enumerate}
\item There exist a model $\pi\colon X_\pi\to X$ of $X$ such that
for all $Q\in X_\pi$ there is a basis $\mathcal B$ of $V$ such that
$$ \pi^{*}(\mathcal B) \ge \mu \pi^{*}D $$
in a Zariski-open neighborhood $U$ of $Q$, relative to the cone of
effective $\mathbb R$-Cartier divisors on $U$.
\item There are finitely many bases $\mathcal B_1,\dots,\mathcal B_\ell$
of $V$; Weil functions $\lambda_{\mathcal B_1},\dots,\lambda_{\mathcal B_\ell}$
for the divisors $(\mathcal B_1),\dots,(\mathcal B_\ell)$, respectively;
a Weil function $\lambda_D$ for $D$; and an $M_k$-constant $c$ such that
\begin{equation}
  \max_{1\le i\le\ell}\lambda_{\mathcal B_i} \ge \mu\lambda_D - c
\end{equation}
(as functions $X(M_k)\to\mathbb R\cup\{+\infty\}$).
\end{enumerate}
Then (i) implies (ii).
\end{proposition}

The proof is the same as in \cite{RV}, except that Lemma \ref{rv_weil_max_fancy}
replaces \cite[Lemma~2.5]{RV} and the step of multiplying by a positive integer
$n$ is omitted.


\section{The Evertse--Ferretti Nevanlinna constant $\NevEF(\mathscr L,D)$}
\label{nev_ef}

This section provides more basic properties of the Evertse--Ferretti
Nevanlinna constant (Definition \ref{def_nevef}), and proves the Main Theorem.

\subsection{Growth conditions}

We first recall the definition of $\mu$-b-growth in \cite{RV}
which appeared in the definition of  $\Nevbir(L,D)$.  This can now be stated
using $\mathbb R$-Cartier divisors instead of $\mathbb Q$-Cartier divisors.

\begin{definition}[{\cite[Def.~4.12]{RV}}]\label{def_mu_b_growth}
Let $X$ be a complete variety, let $D$ be an effective
$\mathbb R$-Cartier divisor
on $X$, let $\mathscr L$ be a line sheaf on $X$, let $V$ be a linear subspace
of $H^0(X,\mathscr L)$ with $\dim V>1$, and let $\mu>0$ be a real number.
We say that $D$ has \textbf{$\mu$-b-growth with respect to $V$ and $\mathscr L$}
if there is a model $\phi\colon Y\to X$ of $X$ such that for all $Q\in Y$
there is a basis $\mathcal B$ of $V$ such that
\begin{equation}\label{eq_mu_b_growth}
  \phi^{*}(\mathcal B) \ge \mu \phi^{*}D
\end{equation}
in a Zariski-open neighborhood $U$ of $Q$, relative to the cone of
effective $\mathbb R$-divisors on $U$.  Also,
we say that $D$ has \textbf{$\mu$-b-growth with respect to $V$}
if it satisfies the above condition with $\mathscr L=\mathscr O(D)$.
\end{definition}

Basically, $\Nevbir(\mathscr L,D)$, according to its definition,
is the infimum of $(\dim V)/\mu$ over all triples $(N,V,\mu)$ such that
$\mu>0$ and $ND$ has $\mu$-b-growth with respect to $V$ and $\mathscr L^N$.

\begin{remark}
Def.~4.12 of \cite{RV} requires $X$ to be normal, but this is unnecessary
since the definition is birational in nature.
\end{remark}

\begin{remark}
If an effective $\mathbb R$-Cartier divisor $D$ has $\mu$-b-growth
with respect to $V$ and $\mathscr L$, then it also has $\mu'$-b-growth
with respect to $V$ and $\mathscr L$ for all $\mu'<\mu$.  Therefore,
if $D$ is $\mathbb Q$-Cartier (or Cartier), then the value of
$\Nevbir(\mathscr L,D)$ according to Definition \ref{def_nevbir} coincides
with the value according to \cite[Def.~4.12]{RV}, because one
can require $\mu$ to be rational in Definition \ref{def_nevbir}
without changing the value of $\Nevbir(\mathscr L,D)$.
\end{remark}

Here is an equivalent condition for $\mu$-b-growth that will be needed
for the proof of Proposition \ref{prop_b_vs_ef_growth}, which compares
$\mu$-b-growth with $\mu$-EF-growth.  This condition is a special case
of \cite[Prop.~4.13]{RV}.  For a fuller list of equivalent conditions,
see \emph{loc.~cit.}

\begin{proposition}\label{prop_mu_b_growth_equiv_iv}
Let $X$ be a complete variety, let $D$ be an effective
$\mathbb R$-Cartier divisor on $X$,
let $\mathscr L$ be a line sheaf on $X$, let $V$ be a linear subspace
of $H^0(X,\mathscr L)$ with $\dim V>1$, and let $\mu>0$ be a real number.
Then $D$ has $\mu$-b-growth with respect to $V$ and $\mathscr L$
if and only if the following condition is true.

For each place $v\in M_k$ there are finitely many bases
$\mathcal B_1,\dots,\mathcal B_\ell$ of $V$; local Weil functions
$\lambda_{\mathcal B_1,v},\dots,\lambda_{\mathcal B_\ell,v}$
for the divisors $(\mathcal B_1),\dots,(\mathcal B_\ell)$, respectively, at $v$;
a local Weil function $\lambda_{D,v}$ for $D$ at $v$; and a constant $c$
such that
\begin{equation}\label{eq_mu_b_growth_iv}
  \max_{1\le i\le\ell}\lambda_{\mathcal B_i,v} \ge \mu\lambda_{D,v} - c
\end{equation}
(as functions $X(\overline k_v)\to\mathbb R\cup\{+\infty\}$).
\end{proposition}

We similarly rephrase the definition of
the Evertse--Ferretti Nevanlinna constant $\NevEF(\mathscr L,D)$
in terms of a definition that corresponds to that of $\mu$-b-growth,
and is suitable for applying the work of Evertse and Ferretti
\cite{ef_festschrift}.

\begin{definition}\label{def_mu_ef_growth}
Let $X$ be a complete variety, let $D$ be an effective $\mathbb R$-Cartier
divisor on $X$, let $\mathscr L$ be a line sheaf on $X$, and let $\mu>0$ be
a real number.  We say that $D$ has
\textbf{$\mu$-EF-growth with respect to $\mathscr L$} if there is
a model $\phi\colon Y\to X$ of $X$ such that for all $P\in Y$ there exist
a base-point-free linear subspace $V\subseteq H^0(X,\mathscr L)$
with $\dim V=\dim X+1$ and a basis $\mathcal B$ of $V$ such that
\begin{equation}\label{eq_mu_ef_growth}
  \phi^{*}(\mathcal B) \ge \mu \phi^{*}D
\end{equation}
in a Zariski-open neighborhood $U$ of $P$, relative to the cone of
effective $\mathbb R$-divisors on $U$.
Also, we say that $D$ has \textbf{$\mu$-EF-growth}
if it satisfies the above condition with $\mathscr L=\mathscr O(D)$.
\end{definition}

Similarly to what was done with $\mu$-b-growth, we have
the following proposition.

\begin{proposition}\label{prop_nevef_via_growth}
Let $\mathscr L$ and $D$ be as in Definition \ref{def_mu_ef_growth}.
Then the Evertse--Ferretti Nevanlinna constant $\NevEF(\mathscr L,D)$
is the infimum of $(\dim X+1)/\mu$ over all pairs $(N,\mu)$
such that $\mu>0$ is real and $ND$ has $\mu$-b-growth
with respect to $\mathscr L^N$.
\end{proposition}

\begin{remark}\label{remk_EF_D_linear}
Let $\mathscr L$ and $D$ be as in Definition \ref{def_mu_ef_growth},
and let $\mu$ and $r$ be positive real numbers.
Then $rD$ has $(\mu/r)$-EF-growth with respect to $\mathscr L$
if and only if $D$ has $r\mu$-EF-growth with respect to $\mathscr L$, and
$$\NevEF(\mathscr L,rD) = r\NevEF(\mathscr L,D)\;.$$
\end{remark}

In regard to the definition of $\mu$-EF-growth, we have the following
equivalent statements (for notations, see \cite{RV}).  These parallel
the full list of \cite[Prop.~4.13]{RV}.

\begin{proposition}\label{prop_mu_ef_growth_equivs}
Let $X$ be a complete variety, let $D$ be an effective $\mathbb R$-Cartier
divisor on $X$, let $\mathscr L$ be a line sheaf on $X$, and let $\mu>0$ be a
real number.  Then the following are equivalent.
\begin{enumerate}
\item $D$ has $\mu$-EF-growth with respect to $\mathscr L$.
\item There are base-point-free linear subspaces $V_1,\dots,V_\ell$
of $H^0(X,\mathscr L)$, all of dimension $\dim X+1$, and corresponding bases
$\mathcal B_1,\dots,\mathcal B_\ell$ of $V_1,\dots,V_\ell$, respectively,
such that
\begin{equation}\label{eq_mu_ef_growth_ii}
  \bigvee_{i=1}^\ell (\mathcal B_i) \ge \mu D
\end{equation}
relative to the cone of effective $\mathbb R$-Cartier b-divisors.
\item There are base-point-free linear subspaces $V_1,\dots,V_\ell$
of $H^0(X,\mathscr L)$, all of dimension $\dim X+1$;
bases $\mathcal B_1,\dots,\mathcal B_\ell$ of $V_1,\dots,V_\ell$, respectively;
Weil functions $\lambda_{\mathcal B_1},\dots,\lambda_{\mathcal B_\ell}$
for the divisors $(\mathcal B_1),\dots,(\mathcal B_\ell)$, respectively;
a Weil function $\lambda_D$ for $D$; and an $M_k$-constant $c$ such that
\begin{equation}\label{eq_mu_ef_growth_iii}
  \max_{1\le i\le\ell}\lambda_{\mathcal B_i} \ge \mu\lambda_D - c
\end{equation}
(as functions $X(M_k)\to\mathbb R\cup\{+\infty\}$).
\item For each $v\in M_k$ there are base-point-free linear subspaces
$V_1,\dots,V_\ell$ of $H^0(X,\mathscr L)$, all of dimension $\dim X+1$;
bases $\mathcal B_1,\dots,\mathcal B_\ell$ of $V_1,\dots,V_\ell$, respectively;
local Weil functions
$\lambda_{\mathcal B_1,v},\dots,\lambda_{\mathcal B_\ell,v}$ at $v$
for the divisors $(\mathcal B_1),\dots,(\mathcal B_\ell)$, respectively;
a local Weil function $\lambda_{D,v}$ for $D$ at $v$; and a constant $c$
such that
\begin{equation}\label{eq_mu_ef_growth_iv}
  \max_{1\le i\le\ell}\lambda_{\mathcal B_i,v} \ge \mu\lambda_{D,v} - c
\end{equation}
(as functions $X(\overline k_v)\to\mathbb R\cup\{+\infty\}$).
\item The condition of (iv) holds for at least one place $v$.
\end{enumerate}
\end{proposition}

Note that the join in (\ref{eq_mu_ef_growth_ii}) exists, because it
involves only integral Cartier b-divisors.

\begin{proof}
Conditions (ii) and (iii) are equivalent by \cite[Prop.~4.10]{RV},
and (iii)--(v) are equivalent by \cite[Prop.~2.4]{RV}.
The implication (ii)$\implies$(i) follows from \cite[Lemma~4.14]{RV},
with $D_i=(\mathcal B_i)$ for all $i$.  (Note that, in the latter lemma,
one can allow $\mu$ to be real and $D$ to be an $\mathbb R$-Cartier divisor:
when working with $\mathbb R$-divisors in the proof, the step of multiplying
by a positive integer $n$ is unnecessary; also,
the last few lines in the proof work equally well with $\mathbb R$-Cartier
divisors in place of $\mathbb Q$-Cartier divisors.)

To finish the proof, it will suffice to show that (i) implies (ii).

Assume that condition (i) is true.  Let $\phi\colon Y\to X$ be a model
that satisfies the condition of Definition \ref{def_mu_ef_growth}.
By quasi-compactness of $Y$,
we may assume that only finitely many triples $(U,V,\mathcal B)$ occur.
Let $(U_1,V_1,\mathcal B_1),\dots,(U_\ell,V_\ell,\mathcal B_\ell)$
be those triples.  We may assume that $\bigvee(\mathcal B_i)$ is represented
by a Cartier divisor $E$ on $Y$.  Then, for each $i$,
$$E\big|_{U_i} \ge \phi^{*}(\mathcal B_i)\big|_{U_i}
  \ge \mu\phi^{*}D\big|_{U_i}$$
relative to the cone of effective $\mathbb R$-Cartier divisors on $U_i$.
This gives (\ref{eq_mu_ef_growth_ii}).
\end{proof}

\subsection{Mumford's theory of degree of contact}

The proof of the Main Theorem also relies on some results from
Mumford's theory of degree of contact.

In this theory, we let $k$ be a field of characteristic $0$,
let $x_1,\dots,x_q$ be homogeneous coordinates on $\mathbb P^{q-1}_k$,
let $Y\subseteq\mathbb P^{q-1}_k$ be a projective variety,
let $H_Y(m)=\dim_k H^0(Y,\mathscr O(m))$ be the Hilbert function of $Y$,
let $I_Y\subseteq k[x_1,\dots,x_q]$ be the ideal of $Y$,
let $S(Y)=k[x_1,\dots,x_q]/I_Y$ be the homogeneous coordinate ring of $Y$,
and let $\mathbf c=(c_1,\dots,c_q)\in\mathbb R_{\ge0}^q$.
Then (as is standard) for all $m\ge0$ we define
$$S_Y(m,\mathbf c)
  = \max\left(\sum_{i=1}^{H_Y(m)} \mathbf a_i\cdot\mathbf c\right)\;,$$
where the maximum is taken over all sets of monomials
$x^{\mathbf a_1}\dotsm x^{\mathbf a_{H_Y(m)}}$ whose residue classes
modulo $I_Y$ form a basis for $S(Y)_m=H^0(Y,\mathscr O(m))$.

In Mumford's theory it is better to consider only varieties that are
\emph{geometrically integral,} because $H^0(Y,\mathscr L)$ is much harder
to control when $Y$ becomes reducible upon base change to a larger field.
Fortunately, though, theorems about Zariski-dense sets of rational points
are all vacuously true on (integral) varieties that are not
geometrically integral, because of the following simple fact.

\begin{lemma}\label{lemma_non_geom_irred}
Let $Y$ be an (integral) variety over a field $k$ of characteristic $0$.
If $Y(k)$ is Zariski dense, then $Y$ is geometrically integral.
\end{lemma}

\begin{proof}
Since $Y_{\text{reg}}$ is open and Zariski dense, it contains a rational
point.  The existence of such a point implies that $Y$ is geometrically
integral.  Indeed, let $P$ be such a point, and let $k'$ be the algebraic
closure of $k$ in the function field $K(Y)$.  Then $\mathscr O_{Y,P}$ is
regular, hence normal.  Since $k'$ is integral over $k$, $\mathscr O_{Y,P}$
contains $k'$.  Therefore its residue field contains $k'$, so $k=k'$
is algebraically closed in $K(Y)$, and therefore $Y$ is geometrically integral.
\end{proof}

\begin{lemma}\label{S_Y_is_geometric}
Let $k$ and $q$ be as above, let $Y\subseteq\mathbb P^{q-1}_k$ be
a geometrically integral projective variety, and let $E$ be an
extension field of $k$.
Let $Y_E=Y\times_k E$, viewed as a projective variety in $\mathbb P^{q-1}_E$
via the map obtained from $Y\hookrightarrow\mathbb P^{q-1}_k$ by base change.
Then, for all $\mathbf c\in\mathbb R_{\ge0}^q$ and all $m\ge0$,
\begin{equation}
  H_{Y_E}(m) = H_Y(m)
\end{equation}
and
\begin{equation}
  S_{Y_E}(m,\mathbf c) = S_Y(m,\mathbf c)\;.
\end{equation}
\end{lemma}

\begin{proof}
We have $H^0(Y_E,\mathscr O(m))=H^0(Y,\mathscr O(m))\otimes_k E$
(by cohomology and flat base change, for example).
Therefore $H_Y(m)=H_{Y_E}(m)$ for all $m$.
Also $S_Y(m,\mathbf c) = S_{Y_E}(m,\mathbf c)$ because the computation
is the same in both cases.
\end{proof}

\begin{proposition}
Let $k$ and $q$ be as above, let $Y\subseteq\mathbb P^{q-1}_k$ be a
geometrically integral projective variety, and let $n=\dim Y$.
Assume that $Y$ is not contained in any coordinate hyperplane
of $\mathbb P^{q-1}$.  Then the inequality
\begin{equation}\label{contact_ineq}
  c_{j_0}+\dots+c_{j_n}
  \le (n+1)\left(\frac{S_Y(m,\mathbf c)}{mH_Y(m)}\right)(1+O(1/m))
\end{equation}
holds for all $\mathbf c\in\mathbb R_{\ge0}^q$ and all collections
$j_0,\dots,j_n$ of indices such that the linear system on $Y$
generated by $x_{j_0},\dots,x_{j_n}$ is base point free.
Here the implicit constant in $O(1/m)$ depends only on $Y$.
\end{proposition}

\begin{proof}  By Lemma \ref{S_Y_is_geometric} we may assume that $k$ is
algebraically closed.

Let $\mathbf c$ and $j_0,\dots,j_n$ be as in the statement of the theorem,
and let $m\in\mathbb Z_{>0}$.  By the theory of degree of contact
(see \cite[Lemma~3.2]{ru}, which combines \cite[Thm.~4.1]{ef_imrn}
and \cite[Lemma~3.2]{ru_annals}),
\begin{equation}\label{eq_ef_vs_b_1}
  c_{j_0}+\dots+c_{j_n}
  \le \frac{e_Y(\mathbf c)}{\Delta}
  \le \frac{n+1}{m}\left(\frac{S_Y(m,\mathbf c)}{H_Y(m)}
    + (2n+1)\Delta\max_{1\le j\le q} c_j\right)\;.
\end{equation}

We claim that
\begin{equation}\label{bound_max_c_j}
  \max_j c_j \le O\left(\frac{S_Y(m,\mathbf c)}{mH_Y(m)}\right)\;,
\end{equation}
with the implicit constant depending only on $Y$.
Pick $j_0,\dots,j_n$ such that that $c_{j_0}=\max_j c_j$
and $x_{j_1}/x_{j_0},\dots,x_{j_n}/x_{j_0}$ form a transcendence base
for $K(Y)$ over $k$.
Then, for each $m>0$, the monomials in $x_{j_0},\dots,x_{j_n}$ of
total degree $m$ are linearly independent in $H^0(Y,\mathscr O(m))$,
so the set $\{x_{j_0}^{\ell_0}\dotsm x_{j_n}^{\ell_n}:\ell_0+\dots+\ell_n=m\}$
of such monomials can be extended to a basis of $H^0(Y,\mathscr O(m))$.
From the definition of $S_Y(m,\mathbf c)$, and
the fact that $H_Y(m)=\frac\Delta{n!}m^n+O(m^{n-1})$, it then follows that
\[
  \begin{split} S_Y(m,\mathbf c)
    &\ge \sum_{\ell_0+\dots+\ell_n=m} c_{j_0}\ell_0 \\
    &= \frac{c_{j_0}m}{n+1}\binom{n+m}{n} \\
    &= \frac{c_{j_0}m}{n+1}\left(\frac{m^n}{n!} + O(m^{n-1})\right) \\
    &= \frac{c_{j_0}}{n+1}\left(\frac1\Delta + O\left(\frac1m\right)\right)
      mH_Y(m)\;.
  \end{split}
\]
This proves (\ref{bound_max_c_j}).  Combining (\ref{eq_ef_vs_b_1}) and
(\ref{bound_max_c_j}) then gives (\ref{contact_ineq}).
\end{proof}

\subsection{Proof of the Main Theorem}

We are now ready to prove the Main Theorem.  This, in turn, reduces to
the following Proposition, which compares the notions of $\mu$-b-growth
and $\mu$-EF-growth.

\begin{proposition}\label{prop_b_vs_ef_growth}
Let $X$ be a geometrically integral variety over a number field $k$,
let $D$ be an effective $\mathbb R$-Cartier divisor on $X$,
let $\mathscr L$ be a line sheaf on $X$, and let $\mu>0$ be a real number.
Assume that $D$ has $\mu$-EF-growth with respect to $\mathscr L$.
Then, for all $\epsilon>0$, there exist a positive integer $m$,
a real number $\nu$, and a linear subspace $V\subseteq H^0(X,\mathscr L^m)$,
such that $mD$ has $\nu$-b-growth with respect to $V$ and $\mathscr L^m$,
and such that
\begin{equation}\label{eq_b_vs_ef_nu_bound}
  \frac{\dim V}{\nu} \le \frac{\dim X + 1}{\mu} + \epsilon\;.
\end{equation}
\end{proposition}

\begin{proof}
Assume that $D$ has $\mu$-EF-growth with respect to $\mathscr L$.
Let $V_1,\dots,V_\ell$; $\mathcal B_1,\dots,\mathcal B_\ell$;
$\lambda_{\mathcal B_1},\dots,\lambda_{\mathcal B_\ell}$; and $\lambda_D$
be as in condition (iv) of Proposition \ref{prop_mu_ef_growth_equivs}.

Let $\{s_1,\dots,s_q\}$ be the elements of $\bigcup_i\mathcal B_i$.
These sections determine a morphism $\Phi\colon X\to\mathbb P^{q-1}$.
Let $Y$ be the image.  Since the $s_i$ are nonzero sections, $Y$ is not
contained in any coordinate hyperplane.  Also note that $\Phi$ need not be
a closed embedding; in fact, it may happen that $\dim Y<\dim X$.
However, it is true that $Y$ is geometrically integral.
Let $n=\dim Y$ and $\Delta=\deg Y$.

Fix a place $v$ of $k$.

For each $j=1,\dots,q$, choose a local Weil function $\lambda_{(s_j),v}$ for
the divisor $(s_j)$ at $v$.  Since each $(s_j)$ is effective, we may assume
that each $\lambda_{(s_j),v}$ is nonnegative.


We will apply (\ref{contact_ineq}) with
$\mathbf c=(\lambda_{(s_1),v}(x),\dots,\lambda_{(s_q),v}(x))$
for some $x\in X(\overline k_v)$.
By (\ref{eq_mu_ef_growth_iv}), there is an index $i$ such that
\begin{equation}\label{eq_ef_vs_b_2}
  \lambda_{\mathcal B_i,v}(x) \ge \mu\lambda_{D,v}(x) + O(1)\;,
\end{equation}
where the implicit constant does not depend on $x$ (or $i$).
Write $\mathcal B_i=\{s_{j_0},\dots,s_{j_n}\}$.  Then
\begin{equation}\label{eq_ef_vs_b_3}
  c_{j_0}+\dots+c_{j_n}
    = \lambda_{(s_{j_0}),v}(x)+\dots+\lambda_{(s_{j_n}),v}(x)
    = \lambda_{\mathcal B_i,v}(x) + O(1)\;,
\end{equation}
where the implicit constant does not depend on $x$.
Since there are only finitely many possible values for $i$, the constant may
also be taken independent of $i$.

Combining (\ref{contact_ineq}), (\ref{eq_ef_vs_b_3}), and (\ref{eq_ef_vs_b_2})
gives
\[
  \begin{split} S_Y(m,\mathbf c)
    &\ge \frac{mH_Y(m)}{(n+1)(1+O(1/m))}\mu\lambda_{D,v}(x) + O(1) \\
    &= \nu m\lambda_{D,v}(x) + O(1)\;,
  \end{split}
\]
where again the implicit constant does not depend on $x$, and
$$\nu = \frac{\mu H_Y(m)}{(n+1)(1+O(1/m))}\;.$$
Let $V\subseteq H^0(X,\mathscr L^n)$ be the pull-back of $H^0(Y,\mathscr O(m))$.
Then $\dim V=H_Y(m)$.  Since $\dim X\ge n$, we have
\[
  \frac{\dim V}{\nu} = \frac{(n+1)(1+O(1/m))}{\mu}
    \le \frac{\dim X+1}{\mu}\left(1+O\left(\frac1m\right)\right)\;.
\]
Thus (\ref{eq_b_vs_ef_nu_bound}) holds for sufficiently large $m$.

On the other hand, by the definition of $S_Y(m,\mathbf c)$ and
our choice of $\mathbf c$, there are bases $\mathcal B_1,\dots,\mathcal B_r$
of $V$ and corresponding local Weil functions
$\lambda_{\mathcal B_1,v},\dots,\lambda_{\mathcal B_r,v}$ such that
\[
  S_Y(m,\mathbf c) = \max_{1\le i\le r}\lambda_{(\mathcal B_i),v}(\Phi(x))
\]
for all $x\in X(\overline k_v)$.
Thus, after pulling the bases back to $V$ and the local Weil functions back
to $X$, we see that $S_Y(m,\mathbf c)$ equals the left-hand side of
(\ref{eq_mu_b_growth_iv}), and hence $mD$ has $\nu$-b-growth
with respect to $V$ and $\mathscr L^m$.
\end{proof}

This proposition then leads quickly to the proof of the Main Theorem.

\begin{proof}[Proof of the Main Theorem]
We may assume that $\NevEF(\mathscr L,D)<\infty$
(otherwise there is nothing to prove).

Let $\epsilon>0$.  By the definition of $\NevEF$, there is a pair $(N,\mu)$
with $N\in\mathbb Z_{>0}$ and $\mu\in\mathbb Q_{>0}$, such that
$ND$ has $\mu$-EF-growth with respect to $\mathscr L^N$ and such that
$$\frac{\dim X + 1}{\mu} < \NevEF(\mathscr L,D) + \epsilon\;.$$
By Proposition \ref{prop_b_vs_ef_growth}, there exist a positive integer $m$,
a rational number $\nu$, and a linear subspace
$V\subseteq H^0(X,\mathscr L^{mN})$ such that $mND$ has $\nu$-b-growth
with respect to $V$ and $\mathscr L^{mN}$, and such that
$$\frac{\dim V}{\nu} \le \frac{\dim X + 1}{\mu} + \epsilon\;.$$
We then have
$$\Nevbir(\mathscr L,D) \le \frac{\dim V}{\nu}
  < \NevEF(\mathscr L,D) + 2\epsilon\;,$$
and the proof concludes by letting $\epsilon$ go to zero.
\end{proof}

\section{Some applications}\label{apps}

In this section, we recover some known results by computing
$\NevEF(\mathscr L,D)$ and then applying Theorem \ref{ef_thmc}.

Let $X$ be a projective variety of dimension $n$  over a number field $k$ (the analytic case will be similar).  Let $D_1, \dots, D_q$ be effective Cartier divisors on $X$.
In this section, such divisors will be said to be in \textbf{general position}
if for all $I\subseteq\{1,\dots,q\}$, every irreducible component of
$\bigcap_{i\in I} \Supp D_i$ has codimension $|I|$.

\noindent{\bf Application 1}.  We assume that $D_1, \dots, D_q$ are in
general position on $X$. In addition,  we assume that
there exist an ample divisor $A$ on $X$ and positive integers $d_1,\dots,d_q$
such that $D_i$ is linearly equivalent to $d_i A$ for all $i=1, \dots, q$.
By replacing $D_i$ with $(d/d_i)D_i$,
where $d=\lcm\{d_1, \dots, d_q\}$, we may assume that $d_i=1$ for all $i$.
Let $D=D_1+ \cdots +D_q$.
To compute $\NevEF(A, D)$, we take $N$ such that $NA$ is very ample,
so there is a morphism $\phi\colon X\rightarrow {\mathbb P}^m$ and hyperplanes
$H_1, \dots, H_q$ in $\mathbb P^m$ such that $\phi^*H_i=ND_i$ for all $i$.
Since $D_1, \dots, D_q$  are in general position on $X$,
for each point $P\in X$ there are at most $n=\dim X$ divisors
among $\{D_1, \dots, D_q\}$ passing through $P$,
so one may choose distinct $i_0,\dots,i_n\in\{1,\dots,q\}$ such that
$D_{i_0},\dots,D_{i_n}$ includes all of the $D_j$ passing through $P$.
  Then $\phi^*H_{j_i}, i=0, \dots, n,$  (regarding $H$ as a section $H\in H^0({\mathbb P}^m, \mathscr O_{{\mathbb P}^m}(1))$)
   forms a basis of a  subspace $V\subset H^0(X, NA)$ with $\dim V=n+1$
(or, equivalently, the $\mathbb Q$-divisors $\frac{1}{N} \phi^*H_{j_i}$,
$i=0, \dots, n$, are actually integral divisors, and they define
a basis of a  subspace $V\subset H^0(X, A)$ with $\dim V=n+1$).
By the condition on general position, this subspace is base point free.
Furthermore, we have, for every irreducible component $E$ of $D$ with $P\in E$,
$$\frac{1}{\ord_E(D)} \sum_{i=0}^n \ord_E\left(\frac{1}{N} \phi^*H_{j_i}\right)
  \ge 1\;.$$
Hence, from the definition,
$$\NevEF(D)\leq {n+1}\;.$$
Thus we recover the following important theorem of Evertse and Ferretti.

\begin{theorem} [Evertse--Ferretti \cite{ef_festschrift}] \label{thm_ef}
Let $X$ be a projective variety over a number field $k$, and let
$D_1, \dots, D_q$ be effective Cartier divisors on $X$ in general position.
Let $S\subset M_k$ be a finite set of places.
Assume that there exist an ample divisor $A$ on $X$ and positive integers
$d_i$ such that $D_i$ is linearly equivalent to $d_i A$ for $i=1, \dots, q$.
Then, for every $\epsilon >0$,
$$\sum_{i=1}^q \frac{1}{d_i} m_S(x, D_i) \leq (\dim X+1+\epsilon) h_A(x)$$
holds for all $k$-rational points outside a proper Zariski closed
subset of $X$.
\end{theorem}

\noindent{\bf Application 2}.  We now only assume that $D_1, \dots, D_q$ are
in general position on $X$. Let $A$ be an ample Cartier divisor on $X$.
Denote by $\epsilon_{D_j}(A)$ the Seshadri constant  of $D_j$ with respect to $A$, which is defined as
$$\epsilon_{D_j}(A)
  = \sup\{\gamma\in\mathbb R: \text{$A-\gamma D_j$ is nef}\}\;.$$
Noticing that the statement in (\ref{HLeqn}) involves $\epsilon$,
it suffices to prove the inequality
$\sum_{j=1}^q \sum_{v\in S} c_j \lambda_{D_j,v}(x) < (n+1+\epsilon)h_A(x)$
in place of (\ref{HLeqn}), with rational $c_j$ close to
$\epsilon_{D_j}(A)$ chosen such that $A - c_j D_j$ is ample for all $j$.
By passing to $\QQ$-Cartier divisors and replacing $D_j$ with $c_jD_j$
we can further assume that $c_j=1$ for all $j$.

In this case, we compute $\NevEF(A, D)$ for $D=D_1+ \cdots +D_q$.
Take $N$ such that $qNA$ is very ample, $ND_j$ is integral for all $j$,
and $N(A-D_j)$ is base point free for all $j$.
According to Heier and Levin (see \cite[\S3]{HL}), there is a morphism
$\phi: X\rightarrow {\mathbb P}^m$ and hyperplanes $H_1, \dots, H_q$
in $\mathbb P^m$ such that $\phi^*H_j\ge ND_j$ for all $j$
and $\phi^* H_1,\dots,\phi^* H_q$ are in general position.
Then, for every $P\in X$, one can find hyperplanes $H_{j_i}, i=0, \dots, n$,
with $\{i_0,\dots, i_n\}\subset \{1, \dots, q\}$ such that
the collection includes all $\phi^{*}H_j$ passing through $P$,
$\phi^*H_{j_i}\ge ND_{j_i}$ for all $i$, and $\phi^*H_{j_i}, i=1, \dots, n$,
are in general position.  Hence $\phi^*H_{j_i}, i=0, \dots, n$, forms
a basis of a base-point free subspace $V\subset H^0(X, NA)$ with $\dim V=n+1$.
Furthermore,  for every irreducible component $E$ of $D$ with $P\in E$, we have
$$\frac{1}{\ord_E(D)} \sum_{i=0}^n \ord_E\left(\frac{1}{N}\phi^*H_{j_i}\right) \ge 1.$$
Thus
$$\NevEF(A, D)\leq n+1.$$

Then, by Theorem \ref{ef_thmc}, we recover the following theorem.

\begin{theorem}[Heier--Levin \cite{HL}]\label{HL}
Let $X$ be a projective variety of dimension $n$ defined over a number field $k$.  Let $D_1, \dots, D_q$ be effective Cartier divisors on $X$ in general position.
Let $S$ be a finite set of places of $k$.
   Let $A$ be an ample Cartier divisor on $X$. Then, for  $\epsilon > 0$,
 there exists a proper Zariski-closed subset $Z\subset X$ such that for all points $x\in X(k)\setminus Z$,
\begin{equation}\label{HLeqn}
\sum_{j=1}^q \ \epsilon_{D_{j}}(A) m_S(x,  D_{j})
  < (n+1+\epsilon)h_A(x),
\end{equation}
where, for all $j$, $\epsilon_{D_{j}}(A)$ is the Seshadri constant of $D_{j}$
with respect to $A$.
\end{theorem}

This also works for subschemes in general position; see \cite{HL}.

\noindent{\bf Application 3}.
We only assume that $D_1, \dots, D_q$ are in $l$-subgeneral position on $X$.

Recall the definition of $l$-subgeneral position:
Let $V$ be a projective variety
and $X\subset V$ be an irreducible subvariety of dimension $n$.
Cartier divisors $D_1,\dots,D_q$ on $V$ are said to be in
$l$-subgeneral position on $X$ if for every choice $J\subset \{1,\dots,q\}$
with $\# J\leq l+1$,
$$\dim \left((\bigcap _{j\in J} \Supp D_j)\cap X\right) \leq l-\#J.$$

The important tool to deal with $l$-subgeneral position is the following result,
which is originally due to Quang (see Lemma 3.1 in \cite{Quang19}).
\begin{lemma}[\cite{Quang19}]\label{quang}
Let $k$ be a number field. Let $X\subset \PP^M_k$ be a projective variety
of dimension $n$. Let $H_1,\dots, H_{l+1}$ be hyperplanes in $\PP^M_k$
which are in $l$-subgeneral position on $X$ with $l\ge n$.
Let $L_1,\dots,L_{l+1}$ be the normalized linear forms defining
$H_1,\dots,H_{l+1}$ respectively. Then there exist linear forms
$L_1',\dots,L_{n+1}'$ on $\PP^M_k$ such that

(a) $L_1' = L_1$.

(b) For every $t\in \{2,\dots,n+1\}$,
$L'_t \in \operatorname{span}_k(L_2,\dots,L_{l-n+t})$;
i.e., $L'_t$ is a $k$-linear combination of $L_2,\dots, L_{l-n+t}$.

(c) The hyperplanes defined by $L_1',\dots,L_{n+1}'$
are in general position on $X$.
\end{lemma}

Using this lemma, similar to the second case, we can prove that,
under the assumptions that $\epsilon_{D_j}(A)=1$ for $j=1, \dots, q$
and that $D_1, \dots, D_q$  are in $l$-subgeneral position on $X$,
 $$\NevEF(A, D)\leq (l-n+1)(n+1).$$
Thus, we have

\begin{theorem} [He--Ru \cite{HR}]
Let $k$ be a number field and let $S$ be a finite set of places of $k$.
Let $X$ be a projective variety of dimension $n$ over $k$.
Let $D_1, \dots, D_q$ be effective divisors on $X$ in  $l$-subgeneral position with $l\ge n$.
Let $A$ be an ample Cartier divisor on $X$.
Then, for all $\epsilon > 0$,
there is a proper Zariski-closed subset $Z\subset X$ such that the inequality
$$\sum_{j=1}^q  \epsilon_{D_{j}}(A) m_S(x,  D_{j})
  < [(l-n+1)(n+1)+\epsilon]h_A(x)$$
holds for all points $x\in X(k)\setminus Z$,
\end{theorem}

This also works for subschemes in $l$-subgeneral position, by blowing them up;
see \cite{HR}.

\section{An Example of Faltings}\label{falt_sect}

In this section we use the notion $\NevEF(D)$ to recover the proof
of a class of examples of Faltings that appeared in his \emph{Baker's Garden}
article \cite{faltings}.
Recall that these examples consist of irreducible divisors $D$
on $\mathbb P^2$ for which $\mathbb P^2\setminus D$ has only finitely many
integral points over $\mathscr O_{k,S}$, where $\mathscr O_{k,S}$
is the localization of the ring of integers of a number field $k$
away from a finite set $S$ of places of $k$.

Note that none of the applications in Section \ref{apps} used the
birational model in Definition \ref{def_nevef}.  In this section, however,
this model is essential.

The main result, Theorem \ref{falt_dioph}, covers all of Faltings' examples,
yet its proof follows rather directly from Corollary \ref{cor_nev_ef_int_pts}.
Since the latter corollary relies
on Schmidt's Subspace Theorem, it necessarily involves varieties that can
be embedded into semiabelian varieties (actually $\mathbb G_{\mathrm m}^N$).
Thus, Faltings' examples can be viewed as examples where one adds components
to the divisor $D$ to obtain a divisor whose complement can be embedded
into a semiabelian variety, in such a way that
the resulting diophantine inequality
is strong enough to give a useful inequality for the original divisor $D$.

In particular, the Shafarevich conjecture (on semistable abelian varieties
over a given number field with good reduction outside of a given finite set
of places, proved by Faltings in 1983) stands out as presently the only
diophantine result with all the hallmarks of a result proved by Thue's method
(ineffective, but with bounds on the number of counterexamples), but which
has not been proved by Thue's method.  This theorem amounts to showing
finiteness of integral points on $\mathcal A_{g,n}$, which also has no
embedding into a semiabelian variety.  It would be interesting
to know if some variant of the above approach could be used to derive
finiteness of integral points on $\mathcal A_{g,n}$ (and therefore the
Shafarevich conjecture) by methods that ultimately rely on Thue's method.

In this section, we give a revised proof of Faltings' result.  We will
split Faltings' main result into a geometric part and an arithmetic part.
The geometric part (Theorem \ref{falt_constr}) guarantees that examples
with certain properties exist, while
the arithmetic part (Theorem \ref{falt_dioph}) says that in each such example
$\mathbb P^2\setminus D$ has only finitely many
integral points over $\mathscr O_{k,S}$.
We prove only the arithmetic part here, since that is the part that
involves the Evertse--Ferretti Nevanlinna constant.

The first (geometric) part of Faltings' result is stated as follows.

\begin{theorem}[Faltings]\label{falt_constr}
Let $k$ be a field of characteristic zero, and let $X$ be a smooth
geometrically irreducible algebraic surface over $k$.
Then, for all sufficiently positive line sheaves $\mathscr L$ on $X$,
there exists a morphism $f\colon X\to\mathbb P^2$ that satisfies
the following conditions.
\begin{enumerate}
\item  $f^{*}\mathscr O(1)\cong\mathscr L$.
\item  The ramification locus $Z$ of $f$ is smooth and irreducible,
and the ramification index is $2$.
\item  The restriction of $f$ to $Z$ is birational onto its
image $D\subseteq\mathbb P^2$.
\item  $D$ is nonsingular except for cusps and simple double points.
\item  Let $Y\to X\to\mathbb P^2$ denote the Galois closure of
$X\to\mathbb P^2$ (i.e., the normalization of $X$ in the Galois closure
of $K(X)$ over $K(\mathbb P^2)$).  Also let $n=\deg f$.  Then $Y$ is
smooth and its Galois group over $\mathbb P^2$ is the full symmetric group
$S_n$.
\item  The ramification locus of $Y$ over $\mathbb P^2$ is the sum of
distinct conjugate effective divisors $Z_{ij}$, $1\le i<j\le n$.  They have
smooth supports, and are disjoint with the following two exceptions.
Points of $Y$ lying over double points of $D$ are fixed points of a subgroup
$S_2\times S_2$ of $S_n$, and they lie
on $Z_{ij}\cap Z_{\ell m}$ with distinct indices $i,j,\ell,m$.
Points of $Y$ lying over cusps of $D$ are fixed points of a subgroup
$S_3$ of $S_n$, and lie on
$Z_{ij}\cap Z_{i\ell}\cap Z_{j\ell}$.
\end{enumerate}
\end{theorem}

For a proof of this theorem, and also an explicit description of the
``sufficiently positive'' condition on $\mathscr L$,
see Faltings' paper \cite{faltings}.

For convenience, write $Z_{ij}=Z_{ji}$ when $i,j\in\{1,\dots,n\}$ and $i>j$.
Let
$$A_i = \sum_{j\ne i} Z_{ij}
  \qquad\text{and}\qquad M = \sum A_i = \sum_{i\ne j} Z_{ij}\;.$$
Let $L$ be the divisor class of $\mathscr L$ on $X$, and let it also denote
the pull-back of this divisor class to $Y$.  In addition, let $d=\deg D$.
We then have
$$2\sum_{i<j}Z_{ij} = \sum A_i = M \sim dL\;.$$

Using the setup as above, we now use Evertse--Ferretti Nevanlinna constants
together with Corollary \ref{cor_nev_ef_int_pts} to prove
the other part of Faltings' result.

\begin{theorem}\label{falt_dioph}
Let $k$ be a number field and let $S$ be a finite set of places of $k$.
Let $Y$, $n$, $\{Z_{ij}\}_{i<j}$, $\{A_i\}_i$, and $M$ be as in
Theorem \ref{falt_constr} and the discussion following it.
Also let $\alpha$ be a rational number such that $M-\alpha A_i$ is an ample
$\mathbb Q$-divisor for all $i$.  Then:
\begin{enumerate}
\item[(a).]  if $\alpha>6$ then no set of $\mathscr O_{k,S}$-integral points
on $Y\setminus\bigcup Z_{ij}$ is Zariski-dense, and
\item[(b).]  if $\alpha>8$ then every set of $\mathscr O_{k,S}$-integral points
on $Y\setminus\bigcup Z_{ij}$ is finite.
\end{enumerate}
Since $Y\setminus\bigcup Z_{ij}$ is an \'etale cover of
$\mathbb P^2\setminus D$, the above conclusions also hold for
$\mathbb P^2\setminus D$ (see \cite[\S\,4.2]{serre} or
\cite[\S\,5.1]{vojta_lnm}).
\end{theorem}

The first part of the proof of this theorem is the following proposition,
which contains all of the geometry specific to the situation of
Theorem \ref{falt_constr}.

\begin{proposition}\label{falt_main_prop}
Let $k$ be a number field, and let $Y$, $n$, $\{Z_{ij}\}_{i<j}$,
$\{A_i\}_i$, $M$, and $\alpha$ be as in Theorem \ref{falt_dioph}.
Assume that $n\ge4$.  Fix Weil functions $\lambda_{ij}$ for each $Z_{ij}$.
Let $\beta$ be an integer such that $\beta\alpha\in\mathbb Z$ and such that
$\beta M$ and all $\beta(M-\alpha A_i)$ are very ample.
Fix an embedding $Y\hookrightarrow\mathbb P^N_k$ associated
to a complete linear system of $\beta M$, and regard $Y$ as a subvariety
of $\mathbb P^N_k$ via this embedding.  Then
\begin{enumerate}
\item[(a).]  There exist a finite list $H_1,\dots,H_q$ of hyperplanes in
$\mathbb P^N_k$, with associated Weil functions $\lambda_{H_j}$ for all $j$,
and constants $c_v$ for all $v\in M_k$, with the following property.
Let $\mathscr J$ be the collection of
all three-element subsets $J=\{j_0,j_1,j_2\}$ of $\{1,\dots,q\}$ for which
$Y\cap H_{j_0}\cap H_{j_1}\cap H_{j_2}=\emptyset$.
Then $\mathscr J\ne\emptyset$, and the inequality
\begin{equation}\label{eq_2.5.1}
  \max_{J\in\mathscr J}\sum_{j\in J} \lambda_{H_j}(y)
    \ge \beta\alpha \sum_{i<j}\lambda_{ij}(y) - c_v
\end{equation}
holds for all $v\in M_k$ and all $y\in Y(\overline k_v)$ not lying on the
support of any $Z_{ij}$ or on any of the $H_j$.
\item[(b).]  Let $C$ be an integral curve in $Y$, not contained in the support
of any $Z_{ij}$.  Then there exist a finite list $H_1,\dots,H_q$ of hyperplanes,
with associated Weil functions as before, and constants $c_v$
for all $v\in M_k$, with the following property.
Let $\mathscr J$ be the collection of all two-element subsets $J=\{j_0,j_1\}$
of $\{1,\dots,q\}$ for which $C\cap H_{j_0}\cap H_{j_1}=\emptyset$.
Then $\mathscr J\ne\emptyset$, and the inequality
\begin{equation}\label{eq_2.5.2}
  \max_{J\in\mathscr J}\sum_{j\in J} \lambda_{H_j}(y)
    \ge \frac{\beta\alpha}{2} \sum_{i<j}\lambda_{ij}(y) - c_v
\end{equation}
holds for all $v\in M_k$ and for all but finitely many $y\in C(\overline k_v)$.
\end{enumerate}
\end{proposition}

The proof of this proposition, in turn, relies mainly on two lemmas.
These lemmas replace Faltings' computations of ideals associated to indices.

\begin{lemma}\label{lemma_2.5.3}
Let $i,j,\ell,m$ be distinct indices.  Then:
\begin{enumerate}
\item[(a).]  there exist hyperplanes $H_0$, $H_1$, and $H_2$ in $\mathbb P^N_k$,
such that
$$Y\cap H_0\cap H_1\cap H_2=\emptyset$$
and
\begin{equation}\label{eq_2.5.3.1}
  (H_0+H_1+H_2)\big|_Y - \beta\alpha(Z_{ij}+Z_{\ell m})
\end{equation}
is an effective Cartier divisor on $Y$; and
\item[(b).]  given any integral curve $C\subseteq Y$ not contained in
any of the $Z_{ab}$, there are hyperplanes $H_0$ and $H_1$ in $\mathbb P^N_k$,
such that $C\cap H_0\cap H_1=\emptyset$ and
\begin{equation}\label{eq_2.5.3.2}
  (H_0+H_1)\big|_C - Z_{ij}\big|_C
\end{equation}
is an effective Cartier divisor on $C$.
\end{enumerate}
\end{lemma}

\begin{proof}
Let $\sigma_i$ and $\sigma_j$ be the canonical sections of $\mathscr O(A_i)$
and $\mathscr O(A_j)$, respectively.  Then the linear system
$$\sigma_i^{\beta\alpha}\cdot\Gamma(Y,\beta(M-\alpha A_i))
  + \sigma_j^{\beta\alpha}\cdot\Gamma(Y,\beta(M-\alpha A_j))$$
has base locus $\Supp A_i\cap\Supp A_j$, since the first summand has
base locus $\Supp A_i$ and the second has base locus $\Supp A_j$.
This intersection consists of the union of $Z_{ij}$ and finitely many closed
points.  Choose an element of this linear system, sufficiently generic so that
it does not vanish identically on any irreducible component of $Z_{\ell m}$,
and let $H_0$ be the associated hyperplane in $\mathbb P^N_k$.  Then
$H_0\big|_Y - \beta\alpha Z_{ij}$ is an effective divisor.

Similarly let $\sigma_\ell$ and $\sigma_m$ be the canonical sections of
$\mathscr O(A_\ell)$ and $\mathscr O(A_m)$, respectively, and let $H_1$ be the
hyperplane associated to an element of
$$\sigma_\ell^{\beta\alpha}\cdot\Gamma(Y,\beta(M-\alpha A_\ell))
  + \sigma_m^{\beta\alpha}\cdot\Gamma(Y,\beta(M-\alpha A_m))\;,$$
chosen sufficiently generically such that $H_1$ does not contain any
irreducible component of $H_0\cap Y$.
Then $H_1\big|_Y - \beta\alpha Z_{\ell m}$ is effective.

By construction, $Y\cap H_0\cap H_1$ is a finite union of closed points,
so we can let $H_2$ be a hyperplane that avoids those points to ensure that
$Y\cap H_0\cap H_1\cap H_2=\emptyset$.  By construction,
$$\left(H_0\big|_Y - \beta\alpha Z_{ij}\right)
  + \left(H_1\big|_Y - \beta\alpha Z_{\ell m}\right) + H_2\big|_Y$$
is effective, and this is the divisor (\ref{eq_2.5.3.1}).  This proves (a).

For part (b), let $\sigma_i$ be as above, and let $H_0$ be the hyperplane
associated to an element of
$\sigma_i^{\beta\alpha}\cdot\Gamma(Y,\beta(M-\alpha A_i))$,
chosen generically so that $H_0$ does not contain $C$.  Let $H_1$ be a
hyperplane in $\mathbb P^N_k$, chosen so that $C\cap H_0\cap H_1=\emptyset$.
Since $H_0\big|_C - Z_{ij}\big|_C$ is an effective divisor,
so is (\ref{eq_2.5.3.2}).
\end{proof}

\begin{lemma}\label{lemma_2.5.4}
Let $i,j,\ell$ be distinct indices.  Then:
\begin{enumerate}
\item[(a).]  there exist hyperplanes $H_0$, $H_1$, and $H_2$ in $\mathbb P^N_k$,
such that
$$Y\cap H_0\cap H_1\cap H_2=\emptyset$$
and
$$(H_0+H_1+H_2)\big|_Y
  - \beta\alpha(Z_{ij} + Z_{i\ell} + Z_{j\ell})$$
is an effective Cartier divisor on $Y$; and
\item[(b).]  given any integral curve $C\subseteq Y$ not contained in any of
the $Z_{ab}$, there are hyperplanes $H_0$ and $H_1$ in $\mathbb P^N_k$,
such that $C\cap H_0\cap H_1=\emptyset$ and
\begin{equation}\label{eq_2.5.4.1}
  (H_0+H_1)\big|_C - \beta\alpha(Z_{ij} + Z_{i\ell})\big|_C
\end{equation}
is an effective Cartier divisor on $C$.
\end{enumerate}
\end{lemma}

\begin{proof}
Let $\sigma_i$ and $\sigma_j$ be as in the preceding proof.
Choose a section of the linear system
$$\sigma_i^{\beta\alpha}\cdot\Gamma(Y,\beta(M-\alpha A_i))
  + \sigma_j^{\beta\alpha}\cdot\Gamma(Y,\beta(M-\alpha A_j))\;,$$
and let $H_0$ be the associated hyperplane.
Then $H_0\big|_Y-\beta\alpha Z_{ij}$ is effective.
We may assume that the choice of $H_0$ is sufficiently generic
so that $H_0$ does not contain any irreducible component of $A_\ell$.

Next let $\sigma_\ell$ be the canonical section of $\mathscr O(A_\ell)$,
and let $H_1$ be the hyperplane associated to a section of
$$\sigma_\ell^{\beta\alpha}\cdot\Gamma(Y,\beta(M-\alpha A_\ell))\;.$$
Then $H_1\big|_Y - \beta\alpha(Z_{i\ell}+Z_{j\ell})$ is effective.
We may assume that $H_1$ does not contain any irreducible component
of $Y\cap H_0$.

Again, $Y\cap H_0\cap H_1$ consists of finitely many points, and we choose
$H_2$ to be any hyperplane not meeting any of these points.  Part (a)
then concludes as in the previous lemma.

For part (b), let $H_0$ and $H_1$ be the hyperplanes associated to suitably
chosen sections of
$\sigma_i^{\beta\alpha}\cdot\Gamma(Y,\beta(M-\alpha A_i))$ and
$\Gamma(Y,\beta M)$, respectively.  As in the previous lemma, we then have
$C\cap H_0\cap H_1=\emptyset$.  Since
$H_0\big|_C-\beta\alpha(Z_{ij}+Z_{i\ell})\big|_C$ is effective,
so is (\ref{eq_2.5.4.1}).
\end{proof}

\begin{proof}[Proof of Proposition \ref{falt_main_prop}]
First consider part (a) of the proposition.

Fix a place $v\in M_k$.
Apply Lemmas \ref{lemma_2.5.3}a and \ref{lemma_2.5.4}a to all possible
collections $i,j,\ell,m$ and $i,j,\ell$ of indices, respectively.
This involves only finitely many applications, so only finitely many
hyperplanes occur.  Let $H_1,\dots,H_q$ be those hyperplanes.

The conditions in Theorem \ref{falt_constr} on the intersections of
the divisors $Z_{ij}$
imply that there is a constant $C_v$ such that, for each $y\in Y(k)$
not in $\bigcup\Supp Z_{ij}$, one of the following conditions holds.
\begin{enumerate}
\item  $\lambda_{ij}(y)\le C_v$ for all $i$ and $j$;
\item  there are indices $i$ and $j$ such that $\lambda_{ij}(y)>C_v$
but $\lambda_{ab}(y)\le C_v$ in all other cases;
\item  there are distinct indices $i,j,\ell,m$ such that
$\lambda_{ij}(y)>C_v$ and $\lambda_{\ell m}(y)>C_v$
but $\lambda_{ab}(y)\le C_v$ in all other cases; or
\item  there are indices $i,j,\ell$ such that
$\max\{\lambda_{ij}(y), \lambda_{i\ell}(y), \lambda_{j\ell}(y)\}>C_v$,
but $\lambda_{ab}(y)\le C_v$ if $\{a,b\}\nsubseteq\{i,j,\ell\}$.
\end{enumerate}

For case (iii), (\ref{eq_2.5.1}) follows from Lemma \ref{lemma_2.5.3}a,
since one can take
$J$ corresponding to the hyperplanes occurring in the lemma,
and the inequality will then follow from effectivity of (\ref{eq_2.5.3.1}).
Case (ii) follows as a special case of this lemma, since $n\ge 4$.
Case (iv) follows from Lemma \ref{lemma_2.5.4}a, by a similar argument.
Finally, in case (i) there is nothing to prove.  This proves (a).

For part (b), let $H_1,\dots,H_q$ be a finite collection of hyperplanes
occurring in all possible applications of Lemmas \ref{lemma_2.5.3}b
and \ref{lemma_2.5.4}b
with the given curve $C$.  We have cases (i)--(iv) as before.
Cases (ii) and (iii) follow from Lemma \ref{lemma_2.5.3}b, where we may assume
without loss of generality that $\lambda_{ij}(y)\ge\lambda_{\ell m}(y)$
to obtain (\ref{eq_2.5.2}) from effectivity of (\ref{eq_2.5.3.2}).
Similarly, case (iv)
follows from Lemma \ref{lemma_2.5.4}b after a suitable
permutation of the indices, and case (i) is again trivial.
\end{proof}

\begin{remark}
It will not actually be needed in the sequel, but the collections $(c_v)$ of
constants in each part of Proposition \ref{falt_main_prop} are actually
$M_k$-constants.  This follows directly from \cite[Prop.~2.4]{RV}.
\end{remark}

\begin{remark}
It is possible (and, in fact, slightly easier) to write
Proposition \ref{falt_main_prop} in terms of Cartier b-divisors
instead of Weil functions (for notations, see \cite{RV}).  For example, one can replace (\ref{eq_2.5.1}) with
\begin{equation}\label{eq_2_5_1bis}
  \bigvee_{J\in\mathscr J} \sum_{j\in J} H_j \ge \beta\alpha\sum_{i<j} Z_{ij}
\end{equation}
relative to the cone of effective Cartier b-divisors.

In the proof, one would let $\phi\colon Y\to X$ be a model for which the
left-hand side of (\ref{eq_2_5_1bis}) pulls back to an ordinary divisor,
and for each $P\in X$ one would consider four cases:
\begin{enumerate}
\item $P\notin Z_{ij}$ for all $i$, $j$;
\item $P\in Z_{ij}$ for exactly one pair $i,j$;
\item there are distinct indices $i,j,\ell,m$ such that $P\in Z_{ij}$
  and $P\in Z_{\ell m}$, but $P\notin Z_{ab}$ for all other components; and
\item there are indices $i,j,\ell$ such that $P$ lies on at least two of
  $Z_{ij}$, $Z_{i\ell}$, and $Z_{j\ell}$, but $P\notin Z_{ab}$ if
  $\{a,b\}\nsubseteq\{i,j,\ell\}$.
\end{enumerate}
In each case let $U$ be the complement of all $Z_{ab}$ that do not contain $P$.
Then $U$ is an open neighborhood of $P$ in $X$, and there is a set $J$
of indices such that $\sum_{j\in J}H_j-\beta\alpha\sum_{a<b}Z_{ab}$ is
effective on $U$.  This set $J$ is obtained from Lemmas \ref{lemma_2.5.3}
or \ref{lemma_2.5.4}, as appropriate.

We chose to keep the phrasing in terms of Weil functions, however, since
that is the phrasing that will be most convenient for the next step in the
proof of Theorem \ref{falt_dioph}.
\end{remark}
The two parts of Proposition \ref{falt_main_prop} say that
$\beta\sum Z_{ij}=(\beta/2)M$ has $\alpha$-EF-growth and $(\alpha/2)$-EF-growth,
respectively, with respect to $\mathscr O(\beta M)$, and therefore
\[
  \NevEF\left(\beta M,\frac\beta2M\right) \le \frac3\alpha
    \qquad\text{and}\qquad
    \NevEF\left(\beta M,\frac\beta2M\right) \le \frac4\alpha\;,
\]
respectively, on $Y$ and $C$, respectively.  By Remark \ref{remk_EF_D_linear},
these become
\begin{equation}\label{eq_falt_nevef}
  \NevEF(\beta M) \le \frac6\alpha \qquad\text{and}\qquad
    \NevEF(\beta M) \le \frac8\alpha\;,
\end{equation}
respectively.

\bigskip
\begin{proof}[Proof of Theorem \ref{falt_dioph}]
(a).  By (\ref{eq_falt_nevef}) and the assumption $\alpha>6$, we have
$$\NevEF(\beta M) \le \frac6\alpha < 1\;.$$
The result then follows by Corollary \ref{cor_nev_ef_int_pts}.

(b).  Let $Z$ be the Zariski closure of a set of $D$-integral points on $Y$.
By part (a), $Z\ne Y$, so it will suffice to show that no irreducible component
of $Z$ can be a curve.  This holds because, on any curve $C$ in $Y$
not contained in $\Supp M$,
$$\NevEF(\beta M) \le \frac8\alpha < 1\;,$$
and we conclude as before.
\end{proof}

\end{document}